\documentclass[11pt,a4paper]{article}
\usepackage[T1]{fontenc}
\usepackage{lmodern}
\usepackage{amsmath,amssymb,amsthm,mathtools}
\usepackage[margin=26mm]{geometry}
\usepackage{microtype}
\usepackage{enumitem}
\usepackage[hidelinks,pdfusetitle]{hyperref}
\usepackage{fancyhdr}
\setlist[enumerate]{itemsep=3pt,topsep=5pt,leftmargin=*}
\newtheorem{theorem}{Theorem}
\newtheorem{lemma}[theorem]{Lemma}

\theoremstyle{remark}

\newcommand{\PP}{\mathbb P}
\newcommand{\EE}{\mathbb E}
\newcommand{\TT}{\mathbb T_D}
\newcommand{\one}{\mathbf 1}
\newcommand{\BPr}{\mathbb P^{\mathrm{br}}_{N,m}}
\newcommand{\BEx}{\mathbb E^{\mathrm{br}}_{N,m}}
\newcommand{\dd}{\,\mathrm d}
\title{Very sharp distance and range transitions for random walk bridges\\on Ramanujan graphs\\[0.5em]\large\itshape Into the abyss with open eyes}
\author{Itai Benjamini}
\date{}

\begin{document}
\maketitle
\vspace{-1.8em}
\begin{abstract}
For vertex-transitive Ramanujan graphs with logarithmic girth, a simple random walk
bridge of length of order $\log N$, where $N$ is the size of the graph, has a maximum
distance that changes from order $\sqrt{\log N}$ to order $\log N$ in a bounded critical window. We prove this by separating
bridges whose lifts to the regular tree close from those whose lifts do not.
A uniform two-term return estimate determines the probabilities of these
two cases and the real-valued critical center. In the same $O(1)$ window, the normalized
range has a two-point limiting law whose mixture weights vary nontrivially across the window. We also prove a uniform first-return
estimate and supercritical range concentration under the weaker assumption
that the girth tends to infinity.
\end{abstract}

\section{Introduction}

Consider a random walk on a
large, locally tree-like Ramanujan graph and condition it to return to its starting
point after a logarithmic number of steps. There are two fundamentally different
ways for this conditioning to be realized. The walk can behave as it would on the
infinite regular tree: its lift to the universal cover closes, its maximal distance
is only of square-root order in the bridge length, and its range has the tree-bridge
constant. Or the lift can fail to close even though its projection returns to the
starting vertex. Such a path must exploit the global identifications of the finite
graph; logarithmic girth then forces it to travel to logarithmic distance, and its
range has the larger, unconditioned-tree constant.

The striking point is how abruptly the balance between these mechanisms changes.
The probability of a closed lift is governed by the tree return probability, whereas
the competing global return has stationary size of order $N^{-1}$. Since the tree
return at time $m$ is of order
\[
 m^{-3/2}\rho^m,
\]
the two mechanisms balance near a logarithmic time with a
$-\frac32\log\log N$ correction. More importantly, the entire crossover takes
place when $m$ is changed by only $O(1)$ steps. Thus a bridge can change from
square-root height to logarithmic height, and its normalized range can switch
between two distinct constants, inside a bounded lattice window. At a fixed
position in this window neither behavior washes out: the limiting law is a genuine
two-point mixture whose weights vary exponentially with the offset.

For perspective, the $O(1)$ lattice window here is much narrower in absolute
time than the $\sqrt{\log N}$-scale cutoff window for ordinary simple random
walk on Ramanujan graphs proved by Lubetzky and Peres~\cite{LP}, although the
two transitions concern different observables: here we condition on an exact
return and study the geometry and range of the resulting bridge. The range of random-walk
bridges on vertex-transitive graphs, including the regular-tree constant that
appears below, was studied by Benjamini, Izkovsky and Kesten~\cite{BIK}.
For background on Ramanujan and expander graphs, see for example
Hoory, Linial and Wigderson~\cite{HLW}.

\section{Setup and statements}

Let $G_N$ be a finite connected simple vertex-transitive $D$-regular graph
on $N$ vertices, where $D\ge3$ is fixed. Let $P_N$ be its simple random walk
transition matrix and $g_N$ its girth. Write
\[
 \rho=\frac{2\sqrt{D-1}}D,\qquad \gamma=-\log\rho,\qquad
 a_N=\begin{cases}1,&G_N\text{ nonbipartite},\\2,&G_N\text{ bipartite},\end{cases}
 \qquad s_N=\frac{a_N}{N}.
\]
The graph is \emph{Ramanujan} if every eigenvalue other than $1$, and $-1$
in the bipartite case, lies in $[-\rho,\rho]$. All logarithms are natural.
All asymptotic statements are as $N\to\infty$. Uniform statements for
$m\asymp\log N$ are understood on any fixed interval
$b\log N\le m\le B\log N$, where $0<b<B<\infty$.
Put
\[
 u_N(r)=P_N^r(o,o),\quad f_N(r)=\PP_o(\tau_o^+=r),\quad
 t_r=P_{\TT}^r(o,o),\quad f_{\mathbb T}(r)=\PP_o^{\TT}(\tau_o^+=r),
\]
where $\tau_o^+=\inf\{r\ge1:X_r=o\}$. Sums over first-return lengths
run over positive integers. For an even bridge length $m$, let
\[
 \BPr=\PP_o(\,\cdot\mid X_m=o),\qquad
 V_m=|\{X_0,\ldots,X_{m-1}\}|,\qquad W_m=V_m/m.
\]
Our first observable is the maximum graph distance reached by the bridge:
\[
 H_m=\max_{0\le j\le m}d_{G_N}(o,X_j).
\]
The two range constants and the tree return constant are
\begin{equation}\label{eq:constants}
 \alpha=\frac{D-2}{2(D-1)},\qquad
 \beta=\frac{D-2}{D-1},\qquad
 b_D=\frac{2^{3/2}D(D-1)}{\sqrt\pi(D-2)^2}.
\end{equation}

The transition occurs when the tree return probability
$t_m\sim b_Dm^{-3/2}e^{-\gamma m}$ becomes comparable with $s_N=a_N/N$.
The exponential factor gives the first-order time $\gamma^{-1}\log N$.
The prefactor $m^{-3/2}$ moves the balance earlier by
$\frac{3}{2\gamma}\log\log N$. This leads to the real-valued critical center
\begin{equation}\label{eq:center}
 m_N^*=\frac{\log N-\frac32\log\log N+
       \log(b_D\gamma^{3/2}/a_N)}{\gamma}.
\end{equation}
The role of bipartiteness is confined to the trivial spectral contribution at
admissible even times: it changes $s_N$ from $1/N$ to $2/N$.  Accordingly it
changes only the constant-order term in $m_N^*$, through $a_N$, and not the
leading $\gamma^{-1}\log N$, the $-\frac{3}{2\gamma}\log\log N$
correction, or the bounded-window mechanism.
Indeed, using $t_m\sim b_Dm^{-3/2}e^{-\gamma m}$ and $s_N=a_N/N$,
\[
 \frac{t_m}{s_N}
 =(1+o(1))\left(\frac{\log N}{\gamma m}\right)^{3/2}
 e^{-\gamma(m-m_N^*)}
\]
uniformly for $m\asymp\log N$. Thus at $m=m_N^*+O(1)$ the prefactor tends
to one and a fixed change of the admissible bridge length changes the odds by a
fixed multiplicative factor. Since bridge lengths are even, all fixed-offset
limits below are understood along admissible even subsequences.
We write $Z_N=\Theta_{\PP}(b_N)$ if, for every $\varepsilon>0$, there are
$0<c_\varepsilon<C_\varepsilon<\infty$ such that
$\PP(c_\varepsilon b_N\le Z_N\le C_\varepsilon b_N)\ge1-\varepsilon$
for all sufficiently large $N$.

Lift the path to the regular-tree universal cover. The lift may close at its
starting vertex, in which case it is a tree bridge with square-root height.
Alternatively, it may end at a different vertex above the same graph root.
Logarithmic girth forces such a path to reach logarithmic distance. Below and
above the transition, respectively, one of these possibilities has probability
tending to one. Inside the critical window, both occur with positive limiting
probabilities.

For example, when $D=3$ one has $\alpha=1/4$ and $\beta=1/2$.  A bridge
whose lift closes is a bridge on the $3$-regular tree and typically has
height of order $\sqrt m$ and normalized range near $1/4$.  If the projected
bridge closes while its lift does not, logarithmic girth forces the path out
to distance of order $\log N$, and its normalized range is instead near
$1/2$.  The results below quantify how the conditional weights of these two
behaviors exchange dominance in a bounded window.

The square-root assertion concerns typical bridges in the precise
$\Theta_{\PP}$ sense above. Its constants may depend on the desired error
probability; it does not assert a limiting deterministic height or a
corresponding estimate for the mean throughout the whole subcritical window.

\begin{theorem}\label{thm:height}
Assume that $G_N$ is Ramanujan and $g_N\ge c\log N$ for a fixed $c>0$.
Let $m=m_N$ be even and satisfy $m\asymp\log N$.
\begin{enumerate}
\item If $m-m_N^*\to-\infty$, then, under the bridge law,
\[
 H_m=\Theta_{\PP}(\sqrt{\log N}).
\]
\item If $m-m_N^*\to+\infty$, then
\[
 \BPr\left(\left\lfloor\frac{g_N}{2}\right\rfloor
                  \le H_m\le\frac m2\right)\longrightarrow1.
\]
In particular, $H_m$ is between fixed positive multiples of $\log N$
with probability tending to one.
\item For every sequence $r_N$ with
$\sqrt{\log N}\ll r_N\ll\log N$,
\begin{equation}\label{eq:heightweight}
 \BPr(H_m\le r_N)=\frac{t_m}{s_N+t_m}+o(1).
\end{equation}
Consequently, if $m-m_N^*\to x\in\mathbb R$, then
\[
 \BPr(H_m\le r_N)\longrightarrow\frac{1}{1+e^{\gamma x}}.
\]
\end{enumerate}
\end{theorem}

In particular, when $m/\log N\to C\in(0,\infty)$, the first two cases
apply for $C<\gamma^{-1}$ and $C>\gamma^{-1}$, respectively.

The tree term is exactly the unconditioned probability that the lift closes.
For bounded critical offsets, the remaining paths contribute asymptotically
$s_N$. The same two weights appear in the normalized range: its possible
limiting values are the fixed constants $\alpha$ and $\beta$, and the
weights assigned to them change with the bridge length.

\begin{theorem}\label{thm:main}
Assume that $G_N$ is Ramanujan and $g_N\ge c\log N$ for a fixed $c>0$.
Fix $0<B<\infty$.
\begin{enumerate}
\item Uniformly over even integers $2\le m\le B\log N$,
\begin{equation}\label{eq:twoterm}
 u_N(m)=s_N+(1+o(1))t_m.
\end{equation}
(The exceptional value $m=0$ satisfies $u_N(0)=t_0=1$ exactly.)
\item If $m=m_N\to\infty$ is even and $m\le B\log N$, set
\[
 \theta_N=\frac{t_m}{s_N+t_m}.
\]
For each fixed integer $k\ge1$,
\begin{equation}\label{eq:moments}
 \BEx[W_m^k]=\theta_N\alpha^k+(1-\theta_N)\beta^k+o(1).
\end{equation}
Consequently, if $\theta_N\to\theta$, then
\begin{equation}\label{eq:mixture}
 W_m\ \Rightarrow\ \theta\delta_\alpha+(1-\theta)\delta_\beta.
\end{equation}
\end{enumerate}
\end{theorem}

In the bounded window, the range has the formula below. The matching
first-return estimate has stationary coefficient $\beta^2$, reflecting
avoidance of the root near both ends of the bridge.

\begin{theorem}\label{thm:critical}
Under the hypotheses of Theorem~\ref{thm:main}, let $w_N\ge0$ satisfy
$w_N=o(\log N)$. Uniformly over even integers $|m-m_N^*|\le w_N$,
\begin{equation}\label{eq:firstcritical}
 f_N(m)=\beta^2s_N+f_{\mathbb T}(m)+o(t_m)
       =\beta^2s_N+(1+o(1))\alpha^2t_m.
\end{equation}
If $m-m_N^*\to x\in\mathbb R$, then
\begin{equation}\label{eq:criticalmixture}
 W_m\ \Rightarrow\
 \frac{1}{1+e^{\gamma x}}\delta_\alpha+
 \frac{e^{\gamma x}}{1+e^{\gamma x}}\delta_\beta.
\end{equation}
In particular,
\[
 \BEx W_m\longrightarrow
 \frac{\alpha+\beta e^{\gamma x}}{1+e^{\gamma x}}.
\]
If $m-m_N^*\to-\infty$ or $+\infty$, with $|m-m_N^*|=o(\log N)$,
then $W_m$ converges in probability to $\alpha$ or $\beta$, respectively.
\end{theorem}

Since bridge lengths are even, the fixed-offset assertions are along
admissible subsequences. The width is bounded independently of $N$: adding
a fixed even number $r$ of steps multiplies $t_m/s_N$ asymptotically by
$\rho^r$. Thus only a bounded change in length is needed to change the
relative weights by a fixed factor.

The critical range law describes coexistence of two values. At offset
$x=0$, asymptotically half the bridges have normalized range near $\alpha$
and half near $\beta$. The mean varies smoothly with $x$, while these two
values remain fixed. In fact, the first two moments in
Theorem~\ref{thm:main} give, at every fixed finite offset,
\[
 \operatorname{Var}_{\BPr}(W_m)\longrightarrow
 (\beta-\alpha)^2\frac{e^{\gamma x}}{(1+e^{\gamma x})^2}>0.
\]
This positive limiting variance makes the distinction between the mixture
and its mean explicit. The logarithmic-girth hypothesis is used throughout
these critical-window conclusions.

Strictly above the spectral transition, concentration of the range needs
only that the girth tends to infinity.

\begin{theorem}\label{thm:supercritical}
Suppose only that $G_N$ is Ramanujan, $g_N\to\infty$, and
$m_N/\log N\to C\in(\gamma^{-1},\infty)$ through even times. Then
\[
 W_{m_N}\xrightarrow{\PP}\beta.
\]
\end{theorem}

The return estimate follows by polynomial approximation, and lifting turns it
into the distance transition. For the range, we approximate its size by local
visit indicators and compute their joint moments by deleting short return
loops. A uniform bound on intermediate returns controls the approximation
error and also gives the first-return estimate. The value $\alpha$ agrees with the tree-bridge range constant in
\cite{BIK}; the estimates used here are proved below.

\section{Tree returns and polynomial approximation}

The tree probabilities determine both the critical center and the two range
constants: the former through their exponential decay and polynomial prefactor,
the latter through their generating function.

\begin{lemma}\label{lem:tree}
Let $F(z)=\sum_{r\ge1}f_{\mathbb T}(r)z^r$ and
$U(z)=\sum_{r\ge0}t_rz^r$. Then
\begin{equation}\label{eq:treegf}
 F(z)=\frac{D}{2(D-1)}\bigl(1-\sqrt{1-\rho^2z^2}\bigr),\qquad
 U(z)=\frac{2(D-1)}{D-2+D\sqrt{1-\rho^2z^2}}.
\end{equation}
In particular,
\begin{equation}\label{eq:escapes}
 1-F(\rho^{-1})=\alpha,\qquad 1-F(1)=\beta,
\end{equation}
and, as $m\to\infty$ through even integers,
\begin{equation}\label{eq:treeasymptotic}
 t_m\sim b_Dm^{-3/2}\rho^m,\qquad
 f_{\mathbb T}(m)\sim\alpha^2t_m.
\end{equation}
There are positive constants depending only on $D$ such that
\begin{equation}\label{eq:treebounds}
 c_D\frac{\rho^m}{(m+1)^{3/2}}\le t_m\le
 C_D\frac{\rho^m}{(m+1)^{3/2}}\qquad(m\ge0\text{ even}).
\end{equation}
\end{lemma}

\begin{proof}
Let $H(z)$ be the generating function of the time to first hit the parent
from a child. A first-step decomposition gives
\[
 H(z)=\frac zD+\frac{D-1}{D}zH(z)^2.
\]
Solving for the power series with zero constant term and using $F=zH$
gives the first formula in \eqref{eq:treegf}. Renewal gives $U=(1-F)^{-1}$,
which proves the second formula and \eqref{eq:escapes}.

For completeness, put
\[
 c_j=\frac{\binom{2j}{j}}{4^j(2j-1)}\quad(j\ge1),\qquad
 1-\sqrt{1-w}=\sum_{j\ge1}c_jw^j.
\]
Rationalizing $U$ gives
\[
 U(z)=\frac{D\sqrt{1-\rho^2z^2}-(D-2)}{2(1-z^2)}.
\]
Since $\frac D2\sum_{j\ge1}c_j\rho^{2j}=1$, coefficient comparison yields
\[
 f_{\mathbb T}(2n)=\frac{D}{2(D-1)}c_n\rho^{2n}\quad(n\ge1),\qquad
 t_{2n}=\frac D2\sum_{j>n}c_j\rho^{2j}\quad(n\ge0).
\]
Stirling's formula gives $c_j\sim(2\sqrt\pi)^{-1}j^{-3/2}$.
The geometric tail in the last display therefore gives
\[
 t_{2n}\sim\frac{D(D-1)}{\sqrt\pi(D-2)^2}n^{-3/2}\rho^{2n},\qquad
 f_{\mathbb T}(2n)\sim\frac{D}{4(D-1)\sqrt\pi}n^{-3/2}\rho^{2n}.
\]
These prove \eqref{eq:treeasymptotic}; adjusting constants at finitely many
times gives \eqref{eq:treebounds}.
\end{proof}

Let $\mu_N$ be the empirical spectral measure of $P_N$, and let
$\mu_{\mathbb T}$ be the spectral measure of the tree transition operator
at its root. Vertex transitivity implies
\[
 u_N(r)=\frac1N\operatorname{tr}(P_N^r)
       =\int\lambda^r\dd\mu_N(\lambda),\qquad
 t_r=\int\lambda^r\dd\mu_{\mathbb T}(\lambda).
\]
The tree measure is supported on $[-\rho,\rho]$. One elementary way to
obtain this support bound is to use the positive weight
$h(x)=(D-1)^{-d(o,x)/2}$. Its neighbor sum is $2\sqrt{D-1}\,h(x)$ away
from the root and no larger at the root. The weighted Schur bound gives
$\|P_{\mathbb T}\|\le\rho$.

For $r<g_N$, a closed walk on $G_N$ lifts to a closed tree walk: otherwise,
after deleting immediate reversals, it would contain a cycle of length
at most $r$. Unique path lifting therefore gives $u_N(r)=t_r$, and hence
\begin{equation}\label{eq:momentmatch}
 \int Q\dd\mu_N=\int Q\dd\mu_{\mathbb T}
 \qquad\text{for every polynomial }\deg Q<g_N.
\end{equation}

The same lifting argument preserves returns in each initial segment and
therefore also gives $f_N(r)=f_{\mathbb T}(r)$ for $r<g_N$.

Low-degree moments suffice to estimate returns at logarithmic times.
A high power can be approximated by a polynomial of degree below the girth,
leaving the trivial eigenvalues as a separate contribution.

\begin{lemma}\label{lem:approx}
Assume the hypotheses of Theorem~\ref{thm:main} and fix $0<B<\infty$.
There is a sequence $\varepsilon_N$ with
$\varepsilon_N(\log N)^{3/2}\to0$ such that, for every integer
$1\le r\le B\log N$,
\begin{equation}\label{eq:uniformerror}
 |u_N(r)-\sigma_N(r)-t_r|\le\varepsilon_N\rho^r,
\end{equation}
where the contribution of the trivial eigenvalues is
\[
 \sigma_N(r)=
 \begin{cases}
  1/N,&G_N\text{ nonbipartite},\\
  (1+(-1)^r)/N,&G_N\text{ bipartite}.
 \end{cases}
\]
\end{lemma}

\begin{proof}
Let $S_r$ be a sum of $r$ independent uniform signs, and let $T_j$ be the
$j$th Chebyshev polynomial. For $x=\cos t$,
$\EE\cos(S_rt)=(\cos t)^r$, so
\[
 x^r=\EE T_{|S_r|}(x).
\]
For an integer $K<g_N$, define
\[
 Q_{r,K}(x)=\EE\bigl[T_{|S_r|}(x)\one_{\{|S_r|\le K\}}\bigr].
\]
This polynomial has degree at most $K$. The exponential moment bound
$\EE e^{tS_r}=(\cosh t)^r\le e^{rt^2/2}$ gives
$\PP(|S_r|>K)\le2e^{-K^2/(2r)}$, and hence
\begin{equation}\label{eq:cheberror}
 \sup_{|x|\le1}|x^r-Q_{r,K}(x)|\le2e^{-K^2/(2r)}.
\end{equation}
If $\kappa=\operatorname{arcosh}(\rho^{-1})$, then
$|T_j(\pm\rho^{-1})|\le e^{\kappa j}$, so
\[
 |Q_{r,K}(\pm\rho^{-1})|\le e^{\kappa K}.
\]
Remove the atoms at the trivial eigenvalues from $\mu_N$. The remaining
positive measure, of mass at most one, is supported on $[-\rho,\rho]$.
Apply \eqref{eq:momentmatch} to $Q_{r,K}(\lambda/\rho)$, and then apply
\eqref{eq:cheberror} to this remaining measure and to $\mu_{\mathbb T}$.
The removed atoms have total mass at most $2/N$. Thus
\begin{equation}\label{eq:chebtrace}
 |u_N(r)-\sigma_N(r)-t_r|
 \le\rho^r\left(4e^{-K^2/(2r)}+\frac{2e^{\kappa K}}N\right).
\end{equation}
Take $K=\lfloor(\log N)^{2/3}\rfloor$, which is below $g_N$ for all
sufficiently large $N$, and set
\[
 \varepsilon_N=4e^{-K^2/(2B\log N)}+2N^{-1}e^{\kappa K}.
\]
This satisfies the required decay and proves the lemma.
\end{proof}

At even times $\sigma_N(r)=s_N$. Dividing
\eqref{eq:uniformerror} by \eqref{eq:treebounds} proves
Theorem~\ref{thm:main}(1) for $r\ge2$; the assertion at $r=0$ follows
from $u_N(0)=t_0=1$ and $s_N\to0$.

\section{Maximum distance from the root}

A tree bridge has height of order the square root of its length, while a
finite-graph bridge whose lift does not close must leave a tree ball around
the root. The return estimate then determines their probabilities.

The square-root estimate below also ensures, through its upper tail, that a
tree bridge of logarithmic length stays well inside the radius on which the
covering map preserves distances.

\begin{lemma}\label{lem:treeheight}
Let $H^{\mathbb T}_{2n}$ be the maximum distance from the root of a
$2n$-step simple random walk bridge on $\TT$. There are $c,C>0$, depending
only on $D$, such that, for all $n\ge1$ and $A\ge1$,
\begin{equation}\label{eq:treeheightupper}
 \PP\bigl(H^{\mathbb T}_{2n}\ge A\sqrt{2n}\mid X_{2n}=o\bigr)
 \le Ce^{-cA^2}.
\end{equation}
Moreover,
\begin{equation}\label{eq:treeheightlower}
 \lim_{a\downarrow0}\ \sup_{n\ge1}
 \PP\bigl(H^{\mathbb T}_{2n}\le a\sqrt{2n}\mid X_{2n}=o\bigr)=0.
\end{equation}
In particular, $H^{\mathbb T}_{2n}=\Theta_{\PP}(\sqrt n)$.
\end{lemma}

\begin{proof}
The radial trajectory of a tree bridge is a Dyck path: a nonnegative
nearest-neighbor path of length $2n$ starting and ending at zero. Let
$\mathsf D_n$ denote the uniform law on these paths, let $K$ be the number
of their positive excursions, and let $M$ be their maximum height.
A radial path with $K=k$ is realized by
$D^k(D-1)^{n-k}$ tree walks. Thus its law under the tree bridge, denoted
by $\mathsf T_n$, satisfies
\begin{equation}\label{eq:dyckdensity}
 \frac{\dd\mathsf T_n}{\dd\mathsf D_n}
 =\frac{v^K}{\mathsf E_{\mathsf D_n}v^K},
 \qquad v=\frac{D}{D-1}<2.
\end{equation}

There are $C_n=(n+1)^{-1}\binom{2n}{n}$ Dyck paths, and the number with
$k$ excursions is
\[
 C_{n,k}=\frac{k}{2n-k}\binom{2n-k}{n}\qquad(1\le k\le n).
\]
Indeed, deleting the first and last step of each excursion gives
$C_{n,k}=[z^{n-k}]C(z)^k$, where $C(z)=1+zC(z)^2$; the displayed
coefficient follows by Lagrange inversion. Consequently,
\[
 \mathsf D_n(K=k)
 =\frac{k(n+1)}{2n-k}
       \prod_{j=0}^{k-1}\frac{n-j}{2n-j}
 \le2k\,2^{-k}.
\]
Since $v\le3/2$ and $v^{3/2}<2$, the moments
$\mathsf E_{\mathsf D_n}v^{3K/2}$ are bounded uniformly in $n$.
The denominator in \eqref{eq:dyckdensity} is at least one. H{\"o}lder's
inequality with exponents $3/2$ and $3$ therefore gives, for every event $E$
of radial paths,
\[
 \mathsf T_n(E)
 =\frac{\mathsf E_{\mathsf D_n}[\one_Ev^K]}
        {\mathsf E_{\mathsf D_n}v^K}
 \le \bigl(\mathsf E_{\mathsf D_n}v^{3K/2}\bigr)^{2/3}
       \mathsf D_n(E)^{1/3},
\]
and hence
\begin{equation}\label{eq:dyckholder}
 \mathsf T_n(E)\le C_D\mathsf D_n(E)^{1/3}.
\end{equation}
It remains to bound the height under the uniform Dyck law.

For the upper tail, take a uniform sequence of $n$ up-steps and $n+1$
down-steps, with partial sums $(S_j)_{j=0}^{2n+1}$. This is a simple
symmetric walk conditioned on $S_{2n+1}=-1$. Rotating the increments
immediately after the first global minimum gives the unique cyclic
rotation whose partial sums are nonnegative until the last step. Removing
that last step produces a uniform Dyck path. Indeed, every such Dyck path
has exactly $2n+1$ preimages; the increment sequence cannot have a proper
period since its total sum is $-1$. The height of the resulting path is
at most $\max_j S_j-\min_j S_j+1$.

For an integer $h\ge1$, reflection at the first visit to $h$ gives
\[
 \PP\left(\max_jS_j\ge h\mid S_{2n+1}=-1\right)
 =\frac{\PP(S_{2n+1}=2h+1)}{\PP(S_{2n+1}=-1)}.
\]
For $0\le h\le n$, the last ratio is
\[
 R_h:=\prod_{j=1}^{h}\frac{n+1-j}{n+1+j}
 \le\exp\!\left(-\frac{h(h+1)}{2n+1}\right),
 \qquad R_0=1;
\]
for $h>n$ it is zero. The bound follows from
$\log(1-y)\le-y$ applied to each factor. Reflection at $-h$ gives
the ratio $R_{h-1}$. Hence both tails are bounded by $Ce^{-ch^2/n}$,
with absolute constants, including $h=1$ after adjusting $C$. It follows that
\begin{equation}\label{eq:uniformdyckupper}
 \mathsf D_n(M\ge h)\le Ce^{-ch^2/n}\qquad(h\ge1).
\end{equation}
Equations~\eqref{eq:dyckholder} and \eqref{eq:uniformdyckupper} prove
\eqref{eq:treeheightupper}.

For the lower tail, let $A_h$ be the adjacency matrix of the path on
$\{0,1,\ldots,h\}$. The number of Dyck paths confined to this interval
is $(A_h^{2n})_{0,0}$. The sine eigenvectors of $A_h$ give
\[
 \frac{(A_h^{2n})_{0,0}}{4^n}
 =\frac{2}{h+2}\sum_{j=1}^{h+1}
   \sin^2\!\frac{\pi j}{h+2}\,
   \cos^{2n}\!\frac{\pi j}{h+2}.
\]
Pair the indices $j$ and $h+2-j$, then use
$\sin x\le x$ and $\cos x\le e^{-x^2/2}$ for $0\le x\le\pi/2$.
Since $C_n\ge c4^n(n+1)^{-3/2}$, this yields, when $h+2\le\sqrt n$,
\begin{equation}\label{eq:uniformdycklower}
 \mathsf D_n(M\le h)
 \le C\left(\frac{\sqrt n}{h+2}\right)^3
          \exp\!\left(-\frac{cn}{(h+2)^2}\right).
\end{equation}
Here we used $\sum_{j\ge1}j^2e^{-xj^2}\le C e^{-cx}$ for $x\ge1$.
Take $h=\lfloor a\sqrt{2n}\rfloor$. If $h=0$, the event is empty. If
$h\ge1$ but $h+2>\sqrt n$, then $a\sqrt{2n}\ge1$ and
$\sqrt n<h+2\le a\sqrt{2n}+2$, which for $a\le1/10$ forces $n$ to lie in
a finite set independent of sufficiently small $a$. For each such fixed $n$,
$\lfloor a\sqrt{2n}\rfloor=0$ once $a$ is small enough, so these cases contribute
nothing to the $\limsup_{a\downarrow0}\sup_n$.

It therefore remains to consider $h+2\le\sqrt n$. Whenever $h\ge1$ we also have
$h+2\le3(h+1)\le C a\sqrt n$ after excluding, as above, the finite range where
$a\sqrt{2n}$ is bounded. Equation~\eqref{eq:uniformdycklower} is then at most
$Ca^{-3}e^{-c/a^2}$, uniformly in the remaining $n$. Combining the finite and
large-$n$ regimes and applying \eqref{eq:dyckholder} proves
\eqref{eq:treeheightlower}.
\end{proof}

The geometric distinction between the two return mechanisms is exact and does
not require a spectral bound.

\begin{lemma}\label{lem:lifting}
Lift the walk from $G_N$ to its universal covering tree $\TT$, starting
at a fixed lift $\widetilde o$ of $o$. Write
\[
 \mathcal C_m=\{\widetilde X_m=\widetilde o\},\qquad
 \widetilde H_m=\max_{0\le j\le m}d_{\TT}(\widetilde o,\widetilde X_j),
 \qquad R_N=\left\lfloor\frac{g_N-2}{2}\right\rfloor.
\]
Then
\begin{equation}\label{eq:closedliftweight}
 \BPr(\mathcal C_m)=\frac{t_m}{u_N(m)}.
\end{equation}
Conditional on $\mathcal C_m$, the lifted path has exactly the law of a
tree bridge of length $m$. Furthermore,
\begin{equation}\label{eq:liftgeometry}
 \mathcal C_m^c\ \Longrightarrow\
       \left\lfloor\frac{g_N}{2}\right\rfloor\le H_m\le m/2,
 \qquad
 \widetilde H_m\le R_N\ \Longrightarrow\ H_m=\widetilde H_m.
\end{equation}
\end{lemma}

\begin{proof}
Path lifting sends simple random walk on $G_N$ to simple random walk on
$\TT$. A lifted return to $\widetilde o$ implies a return to $o$.
Thus the numerator of \eqref{eq:closedliftweight} is $t_m$ before
conditioning, and the denominator is $u_N(m)$. The same bijection of
paths proves the conditional-law assertion.

Since $2R_N+1<g_N$, the ball $B_{G_N}(o,R_N)$ is a tree. Indeed, a
non-tree edge in a breadth-first search tree of this ball would produce
a cycle of length at most $2R_N+1$. Every closed walk contained in this
ball therefore has a closed lift. Hence $\mathcal C_m^c$ implies
$H_m\ge R_N+1=\lfloor g_N/2\rfloor$. For every bridge,
$d(o,X_j)\le\min(j,m-j)$, proving $H_m\le m/2$.
The covering map restricts to an isomorphism between the rooted balls of
radius $R_N$ in $\TT$ and $G_N$. It preserves distances from the root
there, proving the second implication in \eqref{eq:liftgeometry}.
\end{proof}

\begin{proof}[Proof of Theorem~\ref{thm:height}]
The return part of Theorem~\ref{thm:main}, already proved above, gives
\begin{equation}\label{eq:heightmixweight}
 p_N:=\BPr(\mathcal C_m)=\frac{t_m}{u_N(m)}
      =\frac{t_m}{s_N+t_m}+o(1)=\theta_N+o(1).
\end{equation}
Since $m\asymp\log N$ and $R_N\asymp\log N$, the two preceding lemmas
give
\begin{equation}\label{eq:heightcoupling}
 \BPr(H_m\ne\widetilde H_m\mid\mathcal C_m)
 \le\PP\bigl(H^{\mathbb T}_m>R_N\mid X_m=o\bigr)
 \le Ce^{-cR_N^2/m}=o(1).
\end{equation}
Here the upper bound $R_N=O(\log N)$ follows, for example, from the
injective tree ball of radius $R_N$, which contains at least
$(D-1)^{R_N}$ vertices. Thus, conditional on $\mathcal C_m$,
$H_m=\Theta_{\PP}(\sqrt m)$.

For $\sqrt{\log N}\ll r_N\ll\log N$, a nonclosed lift has
$H_m\ge\lfloor g_N/2\rfloor>r_N$ for all large $N$. Conditional on a
closed lift, projection cannot increase distance, and
\eqref{eq:treeheightupper} gives $\BPr(H_m\le r_N\mid\mathcal C_m)=1-o(1)$.
Together with \eqref{eq:heightmixweight}, this proves
\eqref{eq:heightweight}.

Uniformly for $m\asymp\log N$, the definition of $m_N^*$ and the tree
return asymptotic give
\begin{equation}\label{eq:heightcenterratio}
 \frac{t_m}{s_N}
 =(1+o(1))\left(\frac{\log N}{\gamma m}\right)^{3/2}
                 e^{-\gamma(m-m_N^*)}.
\end{equation}
The prefactor is bounded above and away from zero. If $m-m_N^*\to-\infty$,
then $\theta_N\to1$, and the conditional square-root estimate transfers
to the full bridge law. If $m-m_N^*\to+\infty$, then $\theta_N\to0$,
and the first implication in \eqref{eq:liftgeometry} holds with
probability tending to one. Finally, when $m-m_N^*\to x\in\mathbb R$,
the prefactor in \eqref{eq:heightcenterratio} tends to one, giving the
stated critical probability.
\end{proof}

\section{Intermediate returns}

For the rest of the proof of Theorems~\ref{thm:main} and
\ref{thm:critical}, assume logarithmic girth and $m\le B\log N$.
All constants in this section may depend on $D$ and $B$.
Put $h_r=\rho^r/(r+1)^{3/2}$. Lemma~\ref{lem:approx}, including at odd
times, implies
\begin{equation}\label{eq:globalupper}
 u_N(r)\le s_N+C h_r\qquad(0\le r\le B\log N).
\end{equation}
At even times, positivity of the nontrivial spectral contribution gives
$u_N(m)\ge s_N$, while projection of tree walks gives $u_N(m)\ge t_m$.
Consequently,
\begin{equation}\label{eq:denominator}
 u_N(m)\ge c_0(s_N+h_m).
\end{equation}

The following estimate rules out substantial returns at intermediate times;
it is used both to compare local visit indicators with the full range and to
pass from returns to first returns.

\begin{lemma}\label{lem:middle}
For every even $m\le B\log N$ and integer $1\le L<m/2$,
\begin{align}
 \sum_{r=L}^{m-L}u_N(r)u_N(m-r)
 &\le C\bigl(ms_N^2+s_N\rho^L+t_mL^{-1/2}\bigr),
 \label{eq:middleraw}\\
 \sum_{r=L}^{m-L}\frac{u_N(r)u_N(m-r)}{u_N(m)}
 &\le C\left(L^{-1/2}+\frac mN\right).
 \label{eq:middleconditional}
\end{align}
\end{lemma}

\begin{proof}
For $r\le m/2$,
\[
 \frac{h_rh_{m-r}}{h_m}
 =\frac{(m+1)^{3/2}}{(r+1)^{3/2}(m-r+1)^{3/2}}
 \le\frac{2^{3/2}}{(r+1)^{3/2}}.
\]
Splitting at $m/2$ gives
$\sum_{r=L}^{m-L}h_rh_{m-r}\le C h_mL^{-1/2}$.
Also $\sum_{r\ge L}h_r\le C\rho^L$.
Expanding the product in \eqref{eq:globalupper} and using
\eqref{eq:treebounds} proves \eqref{eq:middleraw}.
Divide by \eqref{eq:denominator}; the three terms are bounded by
$Cm/N$, $C\rho^L$, and $CL^{-1/2}$, respectively. This proves
\eqref{eq:middleconditional}.
\end{proof}

Vertex transitivity and the Markov property give
\begin{equation}\label{eq:pair}
 \BPr(X_i=X_j)=\frac{u_N(r)u_N(m-r)}{u_N(m)},
 \qquad r=j-i,\quad 0\le i<j<m.
\end{equation}
Indeed, summing over the common vertex gives
\[
 \sum_xP_N^i(o,x)P_N^r(x,x)P_N^{m-j}(x,o)
 =u_N(r)u_N(m-r)
\]
before division by $u_N(m)$. In particular,
\eqref{eq:middleconditional} bounds the conditional expected number of
visits to $o$ at times between $L$ and $m-L$. It also bounds the probability
of a first return in that interval.

\section{The range law}

We prove Theorem~\ref{thm:main}(2). The observables used below depend only
on equalities among visited vertices. By vertex transitivity, their bridge
law is unchanged if the root is first chosen uniformly. Under this law,
each closed rooted walk of length $m$ has probability
$D^{-m}/(Nu_N(m))$. The law is therefore invariant under cyclic shifts.
We may regard all time indices as elements of $\mathbb Z/m\mathbb Z$;
\eqref{eq:pair} then holds for cyclic separations as well.

For $1\le L<m/2$, define
\[
 Q_i=\one_{\{X_i\notin\{X_{i+1},\ldots,X_{i+L}\}\}},\qquad
 Y_{m,L}=\frac1m\sum_{i=0}^{m-1}Q_i.
\]

Each $Q_i$ marks a visit not followed by another visit to the same vertex
within $L$ steps. A vertex usually contributes one such mark. The following
lemma bounds the error by equal-vertex pairs separated by a long time.

\begin{lemma}\label{lem:rangeapprox}
Under the hypotheses of Theorem~\ref{thm:main},
\begin{equation}\label{eq:rangeapprox}
 \BEx|Y_{m,L}-W_m|
 \le \frac Lm+C L^{-1/2}+C\frac mN.
\end{equation}
\end{lemma}

\begin{proof}
For a visited vertex $v$, list the cyclic gaps between consecutive visits
to $v$. These positive gaps sum to $m$. The contribution of $v$ to
$\sum_iQ_i$ is the number $A_v$ of gaps larger than $L$.

If $A_v=1$, its contribution agrees with its contribution to $V_m$.
If $A_v\ge2$, each of its large gaps is strictly smaller than $m-L$, since
another gap is larger than $L$. Thus $A_v-1$ is bounded by the number of
directed equal-vertex pairs at cyclic separations strictly between $L$
and $m-L$. If $A_v=0$, there are at least $m/L$ visits to $v$. Since
there are $m$ visits in total, at most $L$ vertices have $A_v=0$.
Consequently, deterministically,
\[
 \left|\sum_iQ_i-V_m\right|\le L+M_{m,L},\qquad
 M_{m,L}=\sum_{i=0}^{m-1}\sum_{r=L+1}^{m-L-1}
             \one_{\{X_i=X_{i+r}\}}.
\]
Take expectations, apply \eqref{eq:pair} and
\eqref{eq:middleconditional}, and divide by $m$.
\end{proof}

The exact loop-deletion identity below is the combinatorial core of the moment
calculation.

\begin{lemma}[Loop deletion]\label{lem:looperasure}
Fix $L$ and $k$, and choose indices $i_1,\ldots,i_k$ such that the cyclic sets
$\{i_j,i_j+1,\ldots,i_j+L\}$ are pairwise disjoint. Assume $L<g_N$ and
$m>k(L+1)$. Then
\begin{equation}\label{eq:looperasure}
 \BEx\prod_{j=1}^kQ_{i_j}
 =\sum_{A\subseteq[k]}(-1)^{|A|}
   \sum_{\substack{1\le r_j\le L\\j\in A}}
    \left(\prod_{j\in A}f_{\mathbb T}(r_j)\right)
    \frac{u_N(m-\sum_{j\in A}r_j)}{u_N(m)}.
\end{equation}
where the term for $A=\varnothing$ is $1$.
\end{lemma}

\begin{proof}
Expand each $Q_{i_j}$ by inclusion--exclusion according to whether the walk
returns to $X_{i_j}$ during the following $L$ steps, and classify such a return
by its first return length $r_j$. For a fixed subset $A$ and fixed lengths
$(r_j)_{j\in A}$, the selected time windows are disjoint, so the corresponding
first-return segments are disjoint closed subwalks.

It is convenient to remove any ambiguity about cyclic re-rooting. By vertex
transitivity, the bridge expectation based at $o$ is the same as the expectation
obtained by first choosing the root uniformly from $V(G_N)$ and then choosing a
closed walk of length $m$ from that root. Thus we may work with vertex-rooted
closed walks with arbitrary root; cyclic re-rooting is then an exact
weight-preserving operation on the underlying marked closed walk.

Let $c$ be the smallest time in $\{0,\ldots,m-1\}$ (in the fixed integer
labelling of the cycle) that lies outside the union of the selected windows.
Such a time exists because their union has size at most $k(L+1)<m$. Cyclically
re-root the marked closed walk at time $c$ and rotate the marked window indices
by the same amount. This is a deterministic choice depending only on the marked
indices, so the same $c$ is recovered in the inverse construction. Order the
selected windows cyclically from this root and delete their first-return
segments in that order. After earlier deletions, the insertion point of the
$j$th segment is its rotated index minus the total length of the previously
deleted segments. Thus deletion produces a closed walk of length
$m-\sum_{j\in A}r_j$ together with an ordered collection of rooted
first-return loops and the deterministic root mark. Conversely, insert the
loops in reverse order at the adjusted positions and undo the cyclic rotation.
These operations are mutually inverse and preserve simple-random-walk weight;
in particular, no cyclic multiplicity is introduced.

Because $r_j\le L<g_N$, every such first-return loop lifts uniquely to a
first-return loop on the regular tree. At any base vertex its total simple-walk
weight is therefore $f_{\mathbb T}(r_j)$. The remaining closed walk has total
weight $u_N(m-\sum r_j)$. Dividing by the total bridge weight $u_N(m)$ and
summing over the prescribed lengths proves \eqref{eq:looperasure}.
\end{proof}

For every fixed even $r$, Theorem~\ref{thm:main}(1) and
\eqref{eq:treeasymptotic} give, along any sequence $m\to\infty$ under
consideration,
\begin{equation}\label{eq:ratios}
 \frac{u_N(m-r)}{u_N(m)}
 =(1-\theta_N)+\theta_N\rho^{-r}+o(1).
\end{equation}
Only even lengths contribute to \eqref{eq:looperasure}. Set
\begin{equation}\label{eq:truncatedescapes}
 \alpha_L=1-\sum_{r\le L}f_{\mathbb T}(r)\rho^{-r},\qquad
 \beta_L=1-\sum_{r\le L}f_{\mathbb T}(r).
\end{equation}
Substitution of \eqref{eq:ratios} into \eqref{eq:looperasure} yields
\[
 \BEx\prod_{j=1}^kQ_{i_j}
 =\theta_N\alpha_L^k+(1-\theta_N)\beta_L^k+o(1).
\]
The expression is independent of the positions of the disjoint windows.
Among all ordered $k$-tuples, only $O_k(Lm^{k-1})$ have overlapping
windows. Their contribution to the normalized moment is $o(1)$, so
\begin{equation}\label{eq:Ymoments}
 \BEx Y_{m,L}^k
 =\theta_N\alpha_L^k+(1-\theta_N)\beta_L^k+o(1).
\end{equation}
All variables lie in $[0,1]$, and $|x^k-y^k|\le k|x-y|$ there. Thus
\eqref{eq:Ymoments} and Lemma~\ref{lem:rangeapprox} imply, for each fixed $L$,
\begin{align*}
 \limsup_{N\to\infty}
 \left|\BEx W_m^k-\theta_N\alpha^k-(1-\theta_N)\beta^k\right|
 \le{}& CkL^{-1/2}+|\alpha_L^k-\alpha^k|\\
      &+|\beta_L^k-\beta^k|.
\end{align*}
Equation~\eqref{eq:escapes} gives $\alpha_L\to\alpha$ and
$\beta_L\to\beta$. Sending $L\to\infty$ proves \eqref{eq:moments}
without assuming that $\theta_N$ converges. If $\theta_N\to\theta$,
the limiting moments are those of
$\theta\delta_\alpha+(1-\theta)\delta_\beta$. Polynomial approximation
of continuous functions on $[0,1]$ proves \eqref{eq:mixture}.

\section{First returns and the critical center}

For the first-return assertion of Theorem~\ref{thm:critical},
fix $w_N=o(\log N)$ and consider even $|m-m_N^*|\le w_N$.
Choose $B>\gamma^{-1}$ so that $m\le B\log N$ for all large $N$.
Choose a fixed $\eta>0$ such that $\eta<c$ and
$2\eta<\gamma^{-1}$, and put $L=\lfloor\eta\log N\rfloor$.
Since $g_N\ge c\log N$ and $m\sim\gamma^{-1}\log N$ uniformly in the
window, these choices ensure $L<g_N$ and $2L<m$ for all sufficiently large
$N$.  Moreover, by \eqref{eq:criticalratio} below,
\[
 \frac{s_N\rho^L}{t_m}=N^{o(1)}\rho^L
 =N^{-\gamma\eta+o(1)}=o(1),
\]
Thus no dependence of $\eta$ on the $o(1)$ term is required: any fixed $\eta>0$ satisfying the preceding geometric constraints gives this decay uniformly throughout the window.

Let $A_N(m,L)$ be the unconditioned probability of $X_m=o$ with no
return at times $1,\ldots,L$ or $m-L,\ldots,m-1$. A return in the first
interval is classified by its first occurrence; a return in the last
interval is classified by the first occurrence in the reversed walk.
Reversibility and $2L<m$ give
\begin{align}
 A_N(m,L)={}&u_N(m)
 -2\sum_{r\le L}f_{\mathbb T}(r)u_N(m-r)\notag\\
 &+\sum_{r,s\le L}f_{\mathbb T}(r)f_{\mathbb T}(s)
                   u_N(m-r-s).
 \label{eq:endavoid}
\end{align}
Here $f_N(r)=f_{\mathbb T}(r)$ for $r\le L<g_N$.
Let $A_{\mathbb T}(m,L)$ be the same quantity on the tree.
Substituting \eqref{eq:uniformerror} into \eqref{eq:endavoid} gives
\begin{equation}\label{eq:endcompare}
 A_N(m,L)=A_{\mathbb T}(m,L)+s_N\beta_L^2
                         +O(\varepsilon_N\rho^m).
\end{equation}
Indeed, all contributing shifts are even, so every stationary term is
$s_N$. The absolute sum of the error multipliers is at most
$(1+F(\rho^{-1}))^2$, which is bounded.

A path counted by $A_N(m,L)$ but not by $f_N(m)$ must return at some
time strictly between $L$ and $m-L$. The union bound and the Markov
property imply
\[
 0\le A_N(m,L)-f_N(m)
 \le\sum_{r=L+1}^{m-L-1}u_N(r)u_N(m-r).
\]
Likewise, on the tree,
\[
 0\le A_{\mathbb T}(m,L)-f_{\mathbb T}(m)
 \le Ct_mL^{-1/2}.
\]
Thus the intermediate-return estimates control the errors on \emph{both}
sides of the comparison.  Indeed, writing
$E_N=A_N(m,L)-f_N(m)$ and
$E_{\mathbb T}=A_{\mathbb T}(m,L)-f_{\mathbb T}(m)$, equation
\eqref{eq:endcompare} gives
\[
 f_N(m)-f_{\mathbb T}(m)-s_N\beta_L^2
 = O(\varepsilon_N\rho^m)-E_N+E_{\mathbb T}.
\]
Using \eqref{eq:middleraw} to bound $E_N$ therefore yields the required
two-sided estimate
\begin{equation}\label{eq:firsterror}
 \left|f_N(m)-f_{\mathbb T}(m)-s_N\beta_L^2\right|
 \le C\left(t_mL^{-1/2}+s_N\rho^L+ms_N^2+
                         \varepsilon_N\rho^m\right).
\end{equation}
Also $\beta_L-\beta=O(\rho^L)$ by Lemma~\ref{lem:tree}.

From \eqref{eq:center} and \eqref{eq:treeasymptotic}, uniformly in this
window,
\begin{equation}\label{eq:criticalratio}
 \frac{t_m}{s_N}=(1+o(1))e^{-\gamma(m-m_N^*)},\qquad
 \frac{s_N}{t_m}=N^{o(1)}.
\end{equation}
Since $L\asymp\log N$, every error in \eqref{eq:firsterror}, as well as
$s_N|\beta_L^2-\beta^2|$, is $o(t_m)$ uniformly: the respective bounds
after division by $t_m$ are
\[
 O(L^{-1/2}),\quad N^{o(1)}\rho^L,\quad
 O\bigl(N^{o(1)}\log N/N\bigr),\quad
 O\bigl(\varepsilon_N(\log N)^{3/2}\bigr).
\]
This proves \eqref{eq:firstcritical}.

Indeed, \eqref{eq:criticalratio} follows by writing
$m=m_N^*+x_N$, $x_N=o(\log N)$. Then $m\sim\gamma^{-1}\log N$, and
\[
 \frac{b_Dm^{-3/2}e^{-\gamma m}}{a_N/N}
 =\left(\frac{\log N}{\gamma m}\right)^{3/2}e^{-\gamma x_N}
 =(1+o(1))e^{-\gamma x_N}.
\]
If $x_N\to x$, we obtain
$\theta_N\to(1+e^{\gamma x})^{-1}$. If $x_N\to-\infty$ or $+\infty$,
then $\theta_N\to1$ or $0$. The range assertions of
Theorem~\ref{thm:critical} now follow from Theorem~\ref{thm:main}.
This identifies the constant in the centering and proves a bounded
transition window, with the two-point law at every fixed offset.

\section{Supercritical concentration with only diverging girth}

We prove Theorem~\ref{thm:supercritical}. Here logarithmic girth is not
assumed. At even times the spectral decomposition gives
\[
 u_N(m)=s_N+q_N(m),\qquad 0\le q_N(m)\le\rho^m.
\]
Since $m/\log N\to C>\gamma^{-1}$, we have
$mN\rho^m\to0$. In particular, for every fixed even $r$,
\begin{equation}\label{eq:superratios}
 u_N(m)\sim s_N,\qquad \frac{u_N(m-r)}{u_N(m)}\longrightarrow1.
\end{equation}
At all times, $u_N(r)\le s_N+\rho^r$. Expanding this bound and using
$u_N(m)\ge s_N$ gives, for fixed $1\le L<m/2$,
\begin{equation}\label{eq:supermiddle}
 \sum_{r=L}^{m-L}\frac{u_N(r)u_N(m-r)}{u_N(m)}
 \le C\left(\rho^L+\frac mN+mN\rho^m\right).
\end{equation}
The deterministic cyclic-gap argument in Lemma~\ref{lem:rangeapprox}
therefore yields
\[
 \BEx|Y_{m,L}-W_m|
 \le\frac Lm+C\left(\rho^L+\frac mN+mN\rho^m\right).
\]
Fix $L$. Since $g_N\to\infty$, we have $L<g_N$ for all sufficiently
large $N$, so the exact loop-deletion identity \eqref{eq:looperasure} applies.
We first send $N\to\infty$ with this $L$ fixed, and only afterwards send
$L\to\infty$. Using \eqref{eq:superratios} in place of \eqref{eq:ratios}
gives, for $k=1,2$,
\[
 \BEx Y_{m,L}^k\longrightarrow\beta_L^k.
\]
Since all variables are bounded by one and $\beta_L\to\beta$, this
order of limits gives
$\BEx W_m\to\beta$ and $\BEx W_m^2\to\beta^2$.
Hence $\BEx(W_m-\beta)^2\to0$, which proves the claimed convergence in
probability.

\section*{Acknowledgment}
The author acknowledges ChatGPT for writing this note.

\end{document}